\documentclass[a4paper,12pt]{article}
\usepackage[T1]{fontenc}
\usepackage{amsfonts,amsthm,amssymb,mathrsfs,amsmath}
\usepackage{graphicx,graphics}
\usepackage{relsize}
\usepackage[all]{xy}
\usepackage{tikz}
\allowdisplaybreaks

\allowdisplaybreaks

\newtheorem{theorem}{Theorem}[section]
\newtheorem{lemma}[theorem]{Lemma}

\newtheorem{remark}[theorem]{Remark}
\numberwithin{equation}{section}

\newcommand{\ZZ}{\mathbb{Z}}
\newcommand{\NN}{\mathbb{N}}
\newcommand{\CC}{\mathbb{C}}
\newcommand{\RR}{\mathbb{R}}

\newcommand{\Ex}{\mathop{\mathbb{E}}}
\newcommand{\dprod}[2]{\left\langle #1,#2\right\rangle}

\newcommand{\proj}[1]{\left|#1\right\rangle\left\langle#1\right|}

\title{Spectral Statistics for one dimensional Anderson model with unbounded but decaying potential}
\author{Anish Mallick\\ \texttt{anishm@imsc.res.in} \and Dhriti Ranjan Dolai \\ \texttt{dhdolai@mat.uc.cl}}

\date{\today}

\begin{document}

\maketitle

\begin{abstract}
\noindent In this work we study the spectral statistics for Anderson model on $\ell^2(\NN)$ with decaying randomness whose single site distribution  has unbounded support.
Here we consider the operator $H^\omega$ given by $(H^\omega u)_n=u_{n+1}+u_{n-1}+a_n\omega_n u_n$, $a_n\sim n^{-\alpha}$ and $\{\omega_n\}$
are real i.i.d random variables following symmetric distribution $\mu$ with fat tail,  i.e $\mu((-R,R)^c)<\frac{C}{R^\delta}$ for $R\gg 1$, for some  constant $C$.
In case of $\alpha-\frac{1}{\delta}>\frac{1}{2}$, we are able to show that the eigenvalue process in $(-2,2)$ is the clock process.
\end{abstract}

\noindent{\bf AMS 2010 MSC:} 35J10, 82B44, 81Q10,34L20.

\section{Introduction}
We are interested in the eigenvalue statistics for the operator
\begin{equation}\label{opEqn1}
 (H^\omega u)_n=\left\{\begin{matrix}u_{n+1}+u_{n-1}+a_n\omega_nu_n & n>1 \\ u_2+a_1\omega_1u_1 & n=1 \end{matrix}\right.
\end{equation}
on the Hilbert space $\ell^2(\NN)$, where $a_n\sim n^{-\alpha}$ and $\{\omega_n\}_{n\in\NN}$ are independent identically distributed random variables following real symmetric distribution $\mu$ (that is $\int_{-K}^K xd\mu(x)=0$ for any $K>0$) satisfying
\begin{equation}\label{meaCondEq1}
 \mu((-R,R)^c)\leq \frac{C}{R^\delta}\qquad \forall R\gg1
\end{equation}
for some $\delta>0$ and $C>0$. Note that we  are allowing the probability measure to be singular. To study the eigenvalue statistics we will study the random measures $\{\xi^\omega_{E_0,L}\}_L$, defined later at \eqref{meaDef1}, as $L\rightarrow\infty$. In this work we will be restricting to the case $\alpha-\frac{1}{\delta}>\frac{1}{2}$. 
\\\\
\noindent For Anderson tight binding model, eigenvalue statistics was first studied by Molchanov\cite{M1} for one-dimension and
Minami\cite{M2} for higher dimensions. Other similar works includes Germinet-Klopp \cite{GK}, Geisinger \cite{G1}, Dolai-Krishna \cite{DK2}.
In above mentioned works Poisson statistics was shown in pure point regime of the spectrum. The class of models described by \eqref{opEqn1} are not ergodic,
but recently many results are obtained in non-ergodic cases.
Killip-Stoiciu\cite{KS} studied eigenvalue statistics for CMV matrices (which are unitary operators) and showed that i) for $\alpha>\frac{1}{2}$, when randomness decays rapidly, the eigenvalue process is clock, i.e the gaps between consecutive eigenvalues are fixed, ii) for $\alpha=\frac{1}{2}$ the process is $\beta$-ensemble, this process normally appears in random matrix theory, and iii) for $\alpha<\frac{1}{2}$ it is Poisson. 
In their work, they assumed that the variance of the randomness is finite. 
It is also possible  to study the eigenvalue statistics in the case when randomness is increasing. 
 Dolai-Mallick \cite{DA1} showed that the statistics is Poisson for Anderson model over $\ZZ^d$ when single site potential is of the form $\{|n|^\alpha\omega_n\}_{n\in\ZZ^d}$ for $\alpha>2$. 
 Dolai-Krishna \cite{DK1} showed level repulsion for higher dimensional model with decaying randomness.
\\\\
\noindent CMV matrices are representation of unitary operators and so it is expected that this trend would also hold even for one-dimensional Anderson like models. 
One such work is of Avila-Last-Simon \cite{ALS} where they showed quasi-clock behavior for ergodic Jacobi operator in region of absolutely continuous spectrum.
In \cite{KVV} Krichevski-Valk\'{o}-Vir\'{a}g showed the Sine-$\beta$ process for $\alpha=\frac{1}{2}$, when second moment exists. 
Similar work in continuous analogue are done by Kotani \cite{K1}, where gap between eigenvalue is studied and limit laws is obtained for $\alpha>\frac{1}{2}$. 
Kotani-Nakano\cite{KN} and Nakano\cite{NF} showed that the statistics for $\alpha>\frac{1}{2}$ is clock, and for $\alpha=\frac{1}{2}$ it is circular $\beta$-ensemble. 
In these works, the randomness is via Brownian motion on a compact manifold and the decay is implemented as multiplication to the randomness.
In \cite{KQ} Kotani-Quoc studied the distribution of individual eigenvalue in continuous case when the potential is compound Poisson process. \\

\noindent In this work we look at the eigenvalue statistics inside the essential spectrum of $H^\omega$ for $\alpha-\frac{1}{\delta}>\frac{1}{2}$.
When the measure $\mu$ is absolutely continuous, Delyon-Simon-Souillard \cite{DSS} showed that for $\alpha\delta>1$, the essential spectrum of $H^\omega$ is $[-2,2]$,
they also showed that in the case of $\alpha>\frac{1}{2}$ and $\delta>2$ the spectrum is continuous in that region. 
Using a similar arguments as in \cite{DSS} (see remark \ref{remEssSpec}) we can show that the essential spectrum for $H^\omega$ is $[-2,2]$ for any $\mu$ symmetric and following the bound \eqref{meaCondEq1}. \\
Here we study the eigenvalue statistics in the region $(-2,2)$. For $E_0\in\RR$ the sequence of measures $\{\xi^\omega_{L,E_0}\}_L$ is defined by
\begin{equation}\label{meaDef1}
\xi^\omega_{E_0,L}(f)=tr(f(L(H^\omega_L-E_0)))~~\forall f\in C_c(\RR)
\end{equation}
where $H^\omega_L$ is the restriction of $H^\omega$ onto the subspace $\ell^2(\{1,\cdots,L\})$.\\ ~\\
As a side note we can also look at the following operator
$$(\tilde{H}^\omega_L u)_n=\left\{\begin{matrix}u_2+\frac{\omega_1}{L^\alpha} u_1 & n=1\\ u_{n+1}+u_{n-1}+\frac{\omega_n}{L^\alpha} u_n & n\neq 1,L\\ u_{L-1}+\frac{\omega_L}{L^\alpha} u_L & n=L \end{matrix}\right.$$
on $\ell^2(\{1,\cdots,L\})$ and work with the point process
\begin{equation}\label{meaDef2}
\tilde{\xi}^\omega_{E_0,L}(f)=tr(f(L(\tilde{H}^\omega_L-E_0)))\qquad\forall f\in C_c(\RR),
\end{equation}
where $\{\omega_n\}_{n\in\NN}$ are i.i.d random variable following the distribution $\mu$.
The results and the technique of proof are very similar and so we will place the differences in remarks.
\\
The main result of this work is the following:
%%%%%%%%%%%%%%%%%%%%%%%%%%%%%%%%%%%%%%%%%%%%%
\begin{theorem}\label{mainThm}
On the Hilbert space $\ell^2(\NN)$ consider the random operator $H^\omega$ defined by \eqref{opEqn1} for $\alpha-\frac{1}{\delta}>\frac{1}{2}$, and set  $E_0\in(-2,2)$.
Then given any increasing sequence $\{L_n\}_{n\in\NN}$ in $\NN$, there exists a subsequence $\{L_{n_k}\}_k$ such that the point process $\{\xi^{\omega}_{E_0,L_{n_k}}\}_k$ converges to the clock process in distribution.
\end{theorem}
%%%%%%%%%%%%%%%%%%%%%%%%%%%%%%%%%%%%%%%%%%%%%
\begin{remark}
The equivalent statement for $\tilde{H}^\omega_L$ would be that for $\alpha-\frac{1}{\delta}>\frac{1}{2}$, 
given any increasing sequence $\{L_n\}_{n\in\NN}$, there exists a sub-sequence $\{L_{n_k}\}_{k}$ such that the point process $\{\tilde{\xi}^\omega_{E_0,L_{n_k}}\}_k$ converges to clock process in distribution.

In the following, we will define Pr\"{u}fer phase for the first model. For $\tilde{H}^\omega_L$, we can replace all the $a_n$ by $L^{-\alpha}$. 
It should also be noted that in lemma \ref{lem3} and \ref{lem5}, we do not actually use what are $\{a_n\}$, and so they hold for the Pr\"{u}fer phase defined for $\tilde{H}^\omega_L$.
\end{remark}
%%%%%%%%%%%%%%%%%%%%%%%%%%%%%%%%%%%%%%%%%%%%%
\noindent To study eigenvalue statistics, we work with Pr\"{u}fer phase. For coherence, we provide the derivation to obtain the phase 
function in this setup.\\

\noindent For any $E\in\RR$, there exists a real sequence $u^E:=\{u_n^E\}_n$ such that $H^\omega u^E=E u^E$ by 
looking at the transfer matrix (with $u_0=0$)
$$\left(\begin{matrix}u_{n+1}\\u_n\end{matrix}\right)=\left(\begin{matrix}E-a_n\omega_n & -1 \\ 1 & 0\end{matrix}\right)\left(\begin{matrix}u_n\\ 
u_{n-1}\end{matrix}\right).$$
In case $E\in(-2,2)$, there exists $\theta\in(0,\pi)$ such that $E=2\cos\theta$. Using the transformation
$$\left(\begin{matrix}u_n\\u_{n-1}\end{matrix}\right)=\left(\begin{matrix}\cos n\theta & \sin n\theta \\ 
\cos(n-1)\theta & \sin(n-1)\theta\end{matrix}\right) \left(\begin{matrix}v_n \\ w_n\end{matrix}\right),$$
and setting $w_n+i v_n=r_n e^{i \theta_n}$, we get
$$\theta_{n+1}=\theta_n+\Im\ln\left(1-\frac{a_n\omega_n\sin(\theta_n+n\theta)}{\sin\theta}e^{-i(\theta_n+n\theta)}\right).$$
Here we take the principal branch of logarithm with the branch cut $(-\infty,0)$, so $range(\Im\ln)=(-\pi,\pi)$.
As a result of above transformations, we have
\begin{equation}\label{traEq1}
 \left(\begin{matrix}u_n\\u_{n-1}\end{matrix}\right)=r_n\left(\begin{matrix}\sin(\theta_n+n\theta) \\ \sin(\theta_n+(n-1)\theta)\end{matrix}\right).
\end{equation}
So defining $y^\omega_n(\theta)=\theta_n+n\theta$, and using the fact that $u_0=0$, we have the recursion formula (where $y^\omega_1(\theta)=\theta$)
\begin{align}\label{mainRecEq1}
 y^\omega_n(\theta)&=y^\omega_{n-1}(\theta)+\theta+\Im\ln\left(1-\frac{a_{n-1}\omega_{n-1}\sin y^\omega_{n-1}(\theta)}{\sin\theta}e^{-i y^\omega_{n-1}(\theta)}\right)
\end{align}
and also define $\tilde{y}^\omega_n(\theta)=y^\omega_n(\theta)-n\theta$ (In the case of no randomness, we have $\tilde{y}^\omega_n(\theta)=0$). 
It is clear that $\sin y^\omega_{L+1}(\theta)=0$ for $2\cos\theta\in\sigma(H^\omega_L)$, 
because $u_{L+1}=0$. Therefore to identify the eigenvalues of $H^\omega_L$ inside $(-2,2)$, looking at $\theta$ for which $y^\omega_{L+1}(\theta)\in\pi\ZZ$ is enough.

\noindent Let $\{E^{\omega,L}_i\}_{i=1}^L$ denote the eigenvalues of $H^\omega_L$ (arrange in decreasing order).
Then for any $f\in C_c(\RR)$ fixed, there exists $K>0$ such that $supp(f)\subset[-K,K]$, so one only needs to consider the eigenvalues satisfying
$$L|E^{\omega,L}_i-E_0|<K~~\Rightarrow~~ |E^{\omega,L}_i-E_0|<\frac{K}{L}.$$
Since $E_0\in(-2,2)$, we can assume that all the eigenvalues under consideration
are of the form $E^{\omega,L}_i=2\cos\theta^{\omega,L}_i$ where $\theta^{\omega,L}_i\in(0,\pi)$. 
Set $\theta\in(0,\pi)$ such that $E_0=2\cos\theta$, then
\begin{align}\label{appEq1}
L(E^{\omega,L}_i-E_0)&=2L(\cos\theta^{\omega,L}_i-\cos\theta)\nonumber\\
&=-2L\sin\theta (\theta^{\omega,L}_i-\theta)+O\left(\frac{K^2}{L}\right).
\end{align}
Hence define $x^{\omega,L}_i=L(\theta^{\omega,L}_i-\theta)$ (which are arranged in increasing order).
Most of the work done here to show that $x^{\omega,L}_k$ converges to $k\pi+g^\omega_\theta$ in probability.
\\\\
Here we consider $\mu$ with unbounded support (one can also take $\delta$ to be less than $1$),
we are only able to show convergence in probability. To do this, since $\alpha-\frac{1}{\delta}>\frac{1}{2}$,
there exists a sequence of measurable sets $\{B_N\}_N$ (described in \eqref{defSetEq2}) where after a point all the potential
is decaying fast enough and $\mathbb{P}[B_N]\xrightarrow{N\rightarrow\infty}1$. So limit of $\tilde{y}^\omega_n$
can be computed as element of  $L^2(B_N)$ ($L^2$ is important as seen in lemma \ref{lem6} and cannot be proven for any other exponent).
The construction of sequence $\{B_N\}_N$ is done after lemma \ref{lem5} and in lemma \ref{lem6} the convergence rates in $L^2$ are computed.
Lemma \ref{lem3} and \ref{lem5} are necessary to get one-to-one correspondence between $x_i^{\omega,L}$ and $n_i$ (see \eqref{enumEq1} for correspondence) and
they are completely general (does not depend upon randomness). In lemma \ref{lem4} and \ref{lem2} we get the limit.
Finally we combine the results for the proof of the main theorem.

\noindent If the measure $\mu$ is compactly supported or $\delta>2$, the convergence in lemma \ref{lem6} can be shown to be $L^2$ on $\Omega$ itself
and the theorem can be proved for $\alpha>\frac{1}{2}$.
%%%%%%%%%%%%%%%%%%%%%%%%%%%%%%%%%%%%%%%%%%%%%%%%%%%%%%%%%%%%%%%%%%%%%%%%%%%%%%%%%%%%%%%%%%%%%%%%%%%%%%%%%%%%%%%%%%%%%%%%%%%%%%%%%%%%%%%%%%%%%%%%%%%%%%%%%%%%%%%%%%%%%%%%%%%%%%%%
\section{Result}
Following lemma gives the continuity of $\{y^\omega_n(\cdot)\}_n$. This is done by using the fact that even though $\Im\ln$ is discontinuous, 
the singularity of the $\ln$ is never reached. Final expression \eqref{lem3eq2} is basically Lipschitz continuity statement. 
It can be noted that the neighbourhood where continuity is obtained depends only on $\theta$.
\begin{lemma}\label{lem3}
Given $\theta$, for $N\in\NN$ the function $y^\omega_N(\theta+x)$ is continuous in $x$ in a neighbourhood of $0$ for a.e $\omega$.
\end{lemma}
\begin{proof}
Using
\begin{equation*}
|\ln(1-z)-\ln(1-w)|\leq \frac{|z-w|}{\min\{|1-z|,|1-w|\}}
\end{equation*}
for $|z|,|w|<1$, and writing $y^\omega_{n+1}(\cdot)$ as
\begin{align*}
 y^\omega_{n+1}(\eta)&=y^\omega_n(\eta)+\eta+\Im\ln\left(1-\frac{a_n\omega_n}{a_n\omega_n-2i\sin\eta}e^{-2i y^\omega_n(\eta)}\right)\\
 &\qquad+\Im\ln\left(1+i\frac{a_n\omega_n}{2\sin\eta}\right),
\end{align*}
for $|p|,|q|<\gamma$ where $\gamma>0$ is chosen such that that $\min_{|x|<\gamma}|\sin(\theta+x)|>\frac{|\sin\theta|}{2}$, some simple estimation gives
\begin{align*}
&|y^\omega_{n+1}(\theta+p)-y^\omega_{n+1}(\theta+q)|\\
&\qquad\leq \left(3+\frac{2|a_n\omega_n|}{|\sin\theta|}\right)|y^\omega_{n}(\theta+p)-y^\omega_{n}(\theta+q)|\\
&\qquad\qquad+\left(2+\frac{|a_n\omega_n|}{|\sin\theta|}+\frac{4|a_n\omega_n|}{|\sin\theta|^2}\right)|p-q|.
\end{align*}
Using this recursively till $n=1$ and then using $y^\omega_1(\eta)=\eta$, we get
\begin{align}\label{lem3eq2}
&|y^\omega_{N}(\theta+p)-y^\omega_{N}(\theta+q)|\leq\nonumber\\
&\qquad\left[\sum_{n=1}^{N-1}\left(2+\frac{|a_n\omega_n|}{|\sin\theta|}+\frac{4|a_n\omega_n|}{|\sin\theta|^2}\right)\prod_{m=n}^{N-1}\left(3+\frac{2|a_n\omega_n|}{|\sin\theta|}\right)\right] |p-q|.
\end{align}
This completes the proof.

\end{proof}
\noindent The above lemma and properties of Green's function give the next lemma which show that the functions $y^\omega_n$ are increasing.
%%%%%%%%%%%%%%%%%%%%%%%%%%%%%%%%%%%
\begin{lemma}\label{lem5}
The sequence $\{y^\omega_{L+1}(\theta^{\omega,L}_{i})\}_i$ is increasing and is a subset of $\pi\ZZ$.
\end{lemma}
\begin{proof}
We have the recurrence relation (setting $u^{\omega,E}_0=0$ and $u^{\omega,E}_1=1$)
$$u^{\omega,E}_{n+1}=(E-a_n\omega_n)u^{\omega,E}_n-u^{\omega,E}_{n-1}\qquad\forall n>1,$$
which modifies to 
$$w^{\omega,E}_{n+1}=E-a_n\omega_n-\frac{1}{w^{\omega,E}_{n}},$$
where $w^{\omega,E}_n=\frac{u^{\omega,E}_n}{u^{\omega,E}_{n-1}}$. Using \eqref{traEq1}, we get
$$w^{\omega,2\cos\eta}_n=\frac{\sin y^\omega_n(\eta)}{\sin(y^\omega_n(\eta)-\eta)}=\frac{1}{\cos\eta-\sin\eta\cot y^\omega_n(\eta)},$$
and using definition of Green's function 
\begin{equation}\label{lem5eq2}
 \frac{1}{w^{\omega,E}_{L+1}}=\dprod{\delta_L}{(E-H^\omega_L)^{-1}\delta_L}=\sum_k\frac{a^{\omega,L}_k}{E-E_k^{\omega,L}},
\end{equation}
%where $\{E^{\omega,L}_k\}$ are eigenvalues of $H^\omega_L$ and $a^{\omega,L}_k\geq 0$.
%
%Let $\{\theta^{\omega,L}_i\}_i$ arranged in increasing order be such that $\sigma(H^\omega_L)\cap (-2,2)=\{2\cos\theta^{\omega,L}_i:\forall i\}$, hence
where $a^{\omega,L}_k=\left|\dprod{\phi^{\omega,L}_k}{\delta_L}\right|^2$, $\{\phi^{\omega,L}_k\}_k$
are eigenfunctions corresponding to eigenvalues $\{E^{\omega,L}_k\}_k$ of $H^\omega_L$. By definition of $\theta^{\omega,L}_i$ we have
$$w_{L+1}^{\omega,2\cos\theta^{\omega,L}_i}=0=w_{L+1}^{\omega,2\cos\theta^{\omega,L}_{i+1}}.$$
%%%%%%%%%%%%%%%%%%%%%%
\begin{minipage}{0.65\textwidth}
Let us denote the function $f:(\theta_i^{\omega,L},\theta_{i+1}^{\omega,L})\rightarrow\RR$ defined by $f(\eta)=\frac{1}{w_{L+1}^{\omega,2\cos\eta}}$, 
this is an increasing continuous function (using the properties of last expression in \eqref{lem5eq2} and the fact that $\cos$ is decreasing in 
the concerned region) with 
$$\lim_{\eta\rightarrow (\theta^{\omega,L}_i)^{+}}f(\eta)=-\infty~~\&~~\lim_{\eta\rightarrow (\theta^{\omega,L}_{i+1})^{-}}f(\eta)=\infty.$$
\end{minipage}
\begin{minipage}{0.3\textwidth}
\begin{center}
\includegraphics[width=1.8in,keepaspectratio]{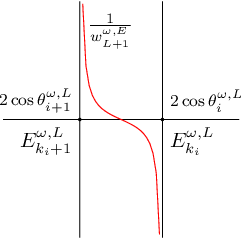} 
\end{center}
\end{minipage}

\noindent Now using the continuity $y^\omega_{L+1}$ and the fact that 
$$\cot y^\omega_{L+1}(\eta)=\frac{\cos\eta-f(\eta)}{\sin\eta}\qquad \forall \eta\in (\theta_i^{\omega,L},\theta_{i+1}^{\omega,L}),$$
observe that
$$y^\omega_{L+1}(\eta)=n^{\omega,L}_i\pi+\cot^{-1}\frac{\cos\eta-f(\eta)}{\sin\eta}~~~ \forall\eta\in(\theta_i^{\omega,L},\theta_{i+1}^{\omega,L}).$$
where $y^\omega_{L+1}(\theta^{\omega,L}_i)=n_i^\omega\pi$. In particular
\begin{equation}\label{lem5eq1}
y^\omega_{L+1}(\theta^{\omega,L}_{i+1})=\lim_{\eta\rightarrow(\theta^{\omega,L}_{i+1})^{-}} y^\omega_{L+1}(\eta)=n^{\omega,L}_i\pi+\pi 
\end{equation}
\end{proof}
\noindent As a consequence of equation \eqref{lem5eq1} and simplicity of spectrum, there exists an enumeration of $\{x_i^{\omega,L}\}_i$ such that
\begin{equation}\label{enumEq1}
 y^{\omega}_{L+1}\left(\theta+\frac{x^{\omega,L}_{i}}{L}\right)=(n^\omega_L+i)\pi
\end{equation}
where $x_0^{\omega,L}$ is defined such that $x^{\omega,L}_{-1}x^{\omega,L}_{1}<0$ (both of them have different sign and none of them are zero). 
From now on we will use this enumeration whenever it is required.\\

\noindent From here we assume that $\alpha-\frac{1}{\delta}>\frac{1}{2}$, so there exists $\beta$ such that $\frac{1}{2}<\beta<\alpha-\frac{1}{\delta}$.
We drop dependence of $\theta$ in notation now onward.

\noindent Define the sequence of sets
\begin{equation}\label{defSetEq1}
\hspace{-0.5cm}A_n:=\{\omega\in\Omega:|a_n\omega_n|<n^{-\beta}|\sin\theta|\}~~~ n\in\NN.
\end{equation}
Using \eqref{meaCondEq1}, we have (choose $N\in\NN$ to be large)
$$\sum_{n\geq 1}\mathbb{P}[A_n^c]\leq N+C|\sin\theta|^\delta\sum_{n\geq N}\frac{1}{n^{\delta(\alpha-\beta)}}<\infty.$$
So using Borel-Canterlli lemma gives $\mathbb{P}[\cap_m\cup_{n\geq m}A_n^c]=0$. So the random variable
\begin{equation}\label{defVarEq1}
n^\omega=\max\{n: |a_n\omega_n|\geq n^{-\beta}|\sin\theta|\}
\end{equation}
is almost everywhere finite, since
$$\{\omega\in\Omega: n^\omega=\infty\}\subseteq \cap_m\cup_{n\geq m}A_n^c.$$
Define the set
\begin{equation}\label{defSetEq2}
B_N=\{\omega\in\Omega: n^\omega< N\},
\end{equation}
and observe that $\mathbb{P}[B_N]\xrightarrow{N\rightarrow\infty}1$. 
%%%%%%%%%%%%%%%%%%%%%%%%%%%%%%%%%%%%%%%%%%%%%%%%%%%%%%
\begin{remark}\label{remEssSpec}
 By the definition of the set $B_N$, it should be clear that for any $\omega\in B_N$, only limit point of the sequence $\{a_n \omega_n\}_n$ is zero. So using the fact $\mathbb{P}[B_N]\xrightarrow{N\rightarrow\infty}1$ we have
 \begin{align*}
 &\mathbb{P}[\omega\in\Omega:\text{limit point set of the sequence $\{a_n \omega_n\}_n$ is }\{0\}]=1,
\end{align*}

 so the operator $V^\omega=\sum_{n\in\NN}a_n\omega_n\proj{\delta_n}$ is almost surely compact, which in turn implies $\sigma_{ess}(H^\omega)=\sigma_{ess}(\Delta)=[-2,2]$ almost surely.
\end{remark}
%%%%%%%%%%%%%%%%%%%%%%%%%%%%%%%%%%%%%%%%%%%%%%%%%%%%%%
Our analysis will be restricted on the sets $\{B_N\}_N$. We  will follow the notation
$$\Ex_{\omega\in S}[X^\omega]:=\int X^\omega\chi_{S}(\omega)d\mathbb{P}(\omega).$$
where $S\subset\Omega$ measurable and $X^\cdot:\Omega\rightarrow\CC$ is a measurable function.
%%%%%%%%%%%%%%%%%%%%%%%%%%%%%%%%%%%%%%%%%%%%%%%%%%%%%%
\begin{lemma}\label{lem6}
For $N$ large enough, there exists $\tilde{C}>0$ such that 
\begin{equation}\label{lem6eq1}
\Ex_{\omega\in B_{N}}\left[\left|\tilde{y}^\omega_M(\theta)-\tilde{y}^\omega_N(\theta)\right|^2\right]\leq \tilde{C} N^{1-2\beta}
\end{equation}
for any $M>N$.
\end{lemma}
\begin{proof}
For $\omega\in B_N$ and $K\geq N$ define
$$Y^\omega_{N,K}(\theta)=\tilde{y}^\omega_N(\theta)+\sum_{n=N}^{K-1} \frac{a_n\omega_n\sin^2 y^\omega_n(\theta)}{\sin\theta}$$
Using $|\ln(1-z)+z|\leq 2|z|^2$ for $|z|<\frac{1}{2}$, we have
\begin{align}\label{lem6eq2}
&\Ex_{\omega\in B_N}\left[\left|\tilde{y}^\omega_M(\theta)-Y^\omega_{N,M}(\theta)\right|^2\right]\\
&\qquad=\Ex_{\omega\in B_N}\left[\left|\sum_{n=N}^{M-1}\left(\Im\ln\left(1-\frac{a_n\omega_n \sin y^\omega_n(\theta)}{\sin\theta}e^{-i y^\omega_n(\theta)}\right)\right.\right.\right.\nonumber\\
&\qquad\qquad\qquad\qquad\qquad\left.\left.\left.-\frac{a_n\omega_n\sin^2 y^\omega_n(\theta)}{\sin\theta}\right)\right|^2\right]\nonumber\\
&\qquad\leq \Ex_{\omega\in B_N}\left[\left(\sum_{n=N}^{M-1}\left|\Im\ln\left(1-\frac{a_n\omega_n \sin y^\omega_n(\theta)}{\sin\theta}e^{-i y^\omega_n(\theta)}\right)\right.\right.\right.\nonumber\\
&\qquad\qquad\qquad\qquad\qquad\left.\left.\left.-\frac{a_n\omega_n\sin^2 y^\omega_n(\theta)}{\sin\theta}\right| \right)^2\right]\nonumber\\
&\qquad\leq 4\Ex_{\omega\in B_N}\left[\left(\sum_{n=N}^{M-1}\left|\frac{a_n\omega_n}{\sin\theta}\right|^2\right)^2\right]\nonumber \\
&\qquad\leq 4\left(\sum_{n=N}^{M-1}\frac{1}{n^{2\beta}}\right)^2\leq C_1 N^{2(1-2\beta)}.
\end{align}
Using $|\sum_i x_i|^2=\sum_{i,j}x_ix_j$ for reals, we have
\begin{align}\label{lem6eq3}
&\Ex_{\omega\in B_N}\left[\left|Y^\omega_{N,M}(\theta)-\tilde{y}^\omega_N(\theta)\right|^2\right]\\
&=\Ex_{\omega\in B_N}\left[\left|\sum_{n=N}^{M-1}\frac{a_n\omega_n\sin^2 y^\omega_n(\theta)}{\sin\theta}\right|^2\right]\nonumber\\
&=\sum_{m,n=N}^{M-1}\Ex_{\omega\in B_N}\left[\frac{a_n\omega_n\sin^2 y^\omega_n(\theta)}{\sin\theta}\frac{a_m\omega_m\sin^2 y^\omega_m(\theta)}
{\sin\theta}\right]\nonumber\\
&=\sum_{n=N}^{M-1}\Ex_{\omega\in B_N}\left[\left(\frac{a_n\omega_n\sin^2 y^\omega_n(\theta)}{\sin\theta}\right)^2\right]\nonumber\\
&~~+\sum_{N\leq m<n<M}2\Ex_{\omega_n\in A_n}[\omega_n]\Ex_{\omega\in B_N}\left[\frac{a_na_m \omega_m \sin^2 y^\omega_n(\theta)\sin^2 y^\omega_m(\theta)}{\sin^2\theta}\right]\nonumber\\
&\leq C_2 N^{1-2\beta}.
\end{align}
Since $y^\omega_n$ is independent of $\omega_n$ and the fact that $\int_{-K}^K xd\mu(x)=0$, we have
\begin{align*}
&\Ex_{\omega\in B_N}\left[\omega_n\omega_m f(y^\omega_n,y^\omega_m)\right]=\\
&\qquad\Ex_{\omega_n\in A_n}[\omega_n]\Ex_{\omega\in B_N}[\omega_m f(y^\omega_n,y^\omega_m)]=0~~\text{for } n>m
\end{align*}

combining \eqref{lem6eq2} and \eqref{lem6eq3} we obtain \eqref{lem6eq1}.
\end{proof}
\begin{remark}
It should if noted that above proof still works if we assume that $\{\omega_n\}_{n\in\NN}$ satisfies
$$\Ex_{\omega\in B_N}[\omega_n X_n^\omega]=0$$
for any $X_n^\omega$ which is a bounded function of $\{\omega_k\}_{k=1}^{n-1}$.

In case $\delta>2$, we can take $\alpha>\frac{1}{2}$ because (define $\tilde{A}_n=\{\omega: |\omega_n|<|a_n||\sin\theta|/2\}$)
$$\sum_{n\geq N}\mathbb{P}[\tilde{A}_n^c]\leq C\sum_{n\geq N}\frac{1}{n^{\alpha\delta}}\xrightarrow{N\rightarrow\infty}0.$$
So $B_N$ is defined as $\{\omega:|\omega_n|<2^{-1}|a_n||\sin\theta|\forall n>N\}$ along with $\beta=\alpha$.
\end{remark}
%%%%%%%%%%%%%%%%%%%%%%%%%%%%%%%%%%%%%%%%%%%%%%%%%%%%%%
\noindent Now we can prove the convergence of $\{\tilde{y}^\omega_n(\theta+\frac{x}{L})\}_n$. This is divided in two parts. 
In the following lemma, convergence of $\{\tilde{y}^\omega_n(\theta)\}_n$ is shown. Then in the next lemma, convergence of 
$\{\tilde{y}^\omega_{L+1}(\theta+\frac{x}{L})-\tilde{y}^\omega_{L+1}(\theta)\}_L$ is shown.
%%%%%%%%%%%%%%%%%%%%%%%%%%%%%%%%%%%%%%%%%%%%%%%%%%%%%%
\begin{lemma}\label{lem2}
For given $\theta$, there exists $g_\theta^\cdot:\Omega\rightarrow [-\infty,\infty]$ measurable, such that for all $\epsilon>0$
\begin{equation}\label{lem2eq1}
\mathbb{P}[\omega\in\Omega: |\tilde{y}^\omega_L(\theta)-g^\omega_\theta|>\epsilon]\xrightarrow{L\rightarrow\infty}0
\end{equation}
\end{lemma}
\begin{proof}
Let $1\ll N<M$, and observe that 
\begin{align*}
&\mathbb{P}[\omega\in\Omega:|\tilde{y}^\omega_M(\theta)-\tilde{y}^\omega_N(\theta)|>\epsilon]\\
&\qquad\leq \mathbb{P}\left[\omega\in B_{N}:|\tilde{y}^\omega_M(\theta)-\tilde{y}^\omega_N(\theta)|>\epsilon\right]+\mathbb{P}[B^c_{N}].
\end{align*}
%%%%
Using Chebyshev's inequality in \eqref{lem6eq1} we get
$$\mathbb{P}\left[\omega\in B_{N}:|\tilde{y}^\omega_M(\theta)-\tilde{y}^\omega_N(\theta)|>\epsilon\right]\leq \frac{\tilde{C}N^{1-2\beta}}{\epsilon^2},$$
which give us
\begin{align*}
 &\mathbb{P}[\omega\in\Omega:|\tilde{y}^\omega_M(\theta)-\tilde{y}^\omega_N(\theta)|>\epsilon]\\
 &\qquad\qquad\leq \mathbb{P}[B^c_{N}]+ \frac{CN^{1-2\beta}}{\epsilon^2}\xrightarrow{N\rightarrow \infty}0,
\end{align*}

So using the fact that Cauchy convergence in probability implies existence of limit we get the measurable function $g^\omega_\theta$ such that \eqref{lem2eq1} hold.\\
\end{proof}
%%%%%%%%%%%%%%%%%%%%%%%%%%%%%%%%%%%%%%%%%%%%%%%%%%%%
\begin{remark}\label{remLabel1}
For $\tilde{H}^\omega_L$ we can set 
$$B_L=\{\omega\in\Omega:|\omega_n|\leq L^\beta|\sin\theta|~\forall 1\leq n\leq L\}$$
and we can show $\mathbb{P}[B_L]\xrightarrow{L\rightarrow\infty}1$. Following the steps of the lemma \ref{lem6} for $\{\tilde{y}^\omega_{L,n}(\theta)\}_{n=1}^{L+1}$ defined for $\tilde{H}^\omega_L$ will give
$$\Ex_{\omega\in B_L}[\left|\tilde{y}^\omega_{L,n}(\theta)\right|^2]\leq \tilde{C}L^{1+2\beta-2\alpha}$$
for all $1\leq n\leq L+1$. So lemma \ref{lem2} modifies to 
\begin{equation}\label{remLab1Eq1}
 \mathbb{P}[\omega\in\Omega:|\tilde{y}^\omega_{L,L+1}(\theta)|>\epsilon]\xrightarrow{L\rightarrow\infty}0.
\end{equation}
So the next lemma is trivially true for $\tilde{y}^\omega_{L,L+1}(\theta)$.
\end{remark}
%%%%%%%%%%%%%%%%%%%%%%%%%%%%%%%%%%%%%%%%%%%%%%%%%%%%
\begin{lemma}\label{lem4}
Given $\theta$ and $K>0$, for $\epsilon>0$ we have
\begin{align}\label{lem4eq1}
\hspace{-1cm}&\lim_{L\rightarrow\infty}\mathbb{P}\left[\omega\in\Omega:\sup_{|x|<K}\left|\tilde{y}^\omega_{L+1}\left(\theta+\frac{x}{L}\right)-\tilde{y}^\omega_{L+1}(\theta)\right|>
\epsilon\right]\nonumber\\
\hspace{-1cm}&\qquad\qquad\qquad\qquad\qquad\qquad=0
\end{align}
\end{lemma}
\begin{proof}
Set $N_L=\lceil(\ln L)^\eta\rceil$ for some $0<\eta<1$, and we have
\begin{align}\label{lem4eq2}
&\mathbb{P}\left[\omega\in\Omega:\sup_{|x|<K}\left|\tilde{y}^\omega_{L+1}\left(\theta+\frac{x}{L}\right)-\tilde{y}^\omega_{L+1}(\theta)\right|>
\epsilon\right]\nonumber\\
&\leq \mathbb{P}\left[\omega\in B_{N_L}:\sup_{|x|<K}\left|\tilde{y}^\omega_{L+1}\left(\theta+\frac{x}{L}\right)-\tilde{y}^\omega_{L+1}(\theta)\right|>
\epsilon\right]+\mathbb{P}[B_{N_L}^c].
\end{align}
Second part converges to zero by definition of $B_N$ (given at \eqref{defSetEq2}). For the first part, observe that
\begin{align}
\label{set}
&\left\{\omega\in B_{N_L}:\sup_{|x|<K}\left|\tilde{y}^\omega_{L+1}\left(\theta+\frac{x}{L}\right)-\tilde{y}^\omega_{L+1}(\theta)\right|>\epsilon\right\}\nonumber\\
&\qquad\subseteq \left\{\omega\in B_{N_L}:\sup_{|x|<K}\left|\tilde{y}^\omega_{L+1}\left(\theta+\frac{x}{L}\right)-\tilde{y}^\omega_{N_L}\left(\theta+\frac{x}{L}\right)\right|>\frac{\epsilon}{3}\right\}\nonumber\\
&\qquad\qquad\bigcup\left\{\omega\in B_{N_L}: \sup_{|x|<K}\left|\tilde{y}^\omega_{N_L}\left(\theta+\frac{x}{L}\right)-\tilde{y}^\omega_{N_L}(\theta)\right|>\frac{\epsilon}{3}\right\}.
\end{align}
Using lemma \ref{lem6} and Chebyshev's inequality for first set on RHS, we have
\begin{align}\label{lem4eq3}
 \mathbb{P}\left[\omega\in B_{N_L}:\sup_{|x|<K} \left|\tilde{y}^\omega_{L}\left(\theta+\frac{x}{L}\right)- \tilde{y}^\omega_{N_L}\left(\theta+\frac{x}{L}\right)\right|>\frac{\epsilon}{3}\right]\leq \tilde{C}\frac{N_L^{1-2\beta}}{\epsilon^2}.
\end{align}
Next we focus on second set in RHS of (\ref{set}). Let $A_n^L=\{\omega:2|a_n\omega_n|<(\ln L)|\sin\theta|\}$ and observe that
\begin{align}\label{lem4eq4}
\mathbb{P}\left[\bigcup_{n=1}^{N_L}(A_n^L)^c\right]&\leq \sum_{n=1}^{N_L}\mathbb{P}\left[(A_n^L)^c\right]\leq \frac{C}{(\ln L)^\delta |\sin\theta|^\delta}\sum_{n=1}^{N_L} \frac{1}{n^{\alpha\delta}}\nonumber\\
&<\frac{C_2}{(\ln L)^\delta}\xrightarrow{L\rightarrow\infty}0.
\end{align}
Setting $D_L=\cap_{n=1}^{N_L}A^L_n$ and we have $\mathbb{P}[D_L]\xrightarrow{L\rightarrow\infty}1$. For $\omega\in D_L$, using \eqref{lem3eq2} we have
\begin{align*}
&\sup_{|x|<K}\left|y^\omega_{N_L}\left(\theta+\frac{x}{L}\right)-y^\omega_{N_L}(\theta)\right|\\
&\qquad\leq \left[\sum_{n=1}^{N_L-1}\left(2+\frac{|a_n\omega_n|}{|\sin\theta|}+\frac{4|a_n\omega_n|}{|\sin\theta|^2}\right)\prod_{m=n}^{N_L-1}\left(3+\frac{2|a_n\omega_n|}{|\sin\theta|}\right)\right]\frac{K}{L}\\
&\qquad\leq C_3K \frac{(\ln L)^{N_L}}{L}=O\left(e^{(\ln L)^\eta \ln\ln L -\ln L}\right)\xrightarrow{L\rightarrow\infty}0.
\end{align*}
Using above
\begin{align*}
&\mathbb{P}\left[\omega\in B_{N_L}:\sup_{|x|<K}\left|\tilde{y}^\omega_{N_L}\left(\theta+\frac{x}{L}\right)-\tilde{y}^\omega_{N_L}(\theta)\right|>\frac{\epsilon}{3}\right]\\
&\leq \mathbb{P}\left[\omega\in D_L:\sup_{|x|<K}\left|y^\omega_{N_L}\left(\theta+\frac{x}{L}\right)-y^\omega_{N_L}(\theta)+\frac{N_L x}{L}\right|>\frac{\epsilon}{3}\right]+\mathbb{P}[D_L^c]\\
&=\mathbb{P}[D_L^c]
\end{align*}
Combining \eqref{lem4eq2},\eqref{lem4eq3} and \eqref{lem4eq4} we have the result.\\
\end{proof}
%%%%%%%%%%%%%%%%%%%%%%%%%%%%%%%%%%%%
\begin{remark}
It should be noted that in case of lemma \ref{lem2} and \ref{lem4}, the convergence is shown for the entire sequence.

The problem that arises in proving clock process for entire sequence is the fact that the limit points of the sequence $\{\{L\frac{\theta}{\pi}\}\}_{L\in\NN}$ ($\{x\}$ is the fractional part of $x$) is all of the interval $[0,1]$ for almost every $\theta\in[0,\pi]$.
\end{remark}
\noindent Before going into proof of main theorem, we choose the subsequence of $\{L_n\}_n$ say $\{L_{n_k}\}_k$ such that 
$$\left\{\frac{(L_{n_k}+1)\theta}{\pi}\right\}\rightarrow a$$ 
for some $a$ which is a limit point of $\{L_n\}_n$.
\section {Proof of main Theorem \ref{mainThm}}
\begin{proof}
Let $f\in C_c(\RR)$ real valued with $supp(f)\subset[-K,K]$. Given $\epsilon>0$, uniform continuity of $f$ implies the existence of $\eta>0$ such that
$$|f(x)-f(y)|<\epsilon\qquad\forall |x-y|<\eta.$$
Combining lemma \ref{lem4} and \ref{lem2} gives
\begin{align*}
 \lim_{L\rightarrow\infty}\mathbb{P}\left[\omega\in\Omega:\sup_{|x|<K}\left|y^\omega_{L+1}\left(\theta+\frac{x}{L}\right)-(L+1)\left(\theta+\frac{x}{L}\right)-g^\omega_\theta\right|>\frac{\eta}{3}\right]=0.
 %&\xrightarrow{L\rightarrow\infty}0.
\end{align*}
Since $K$ is fixed, above statement can be modified to
\begin{align*}
\lim_{L\rightarrow\infty}\mathbb{P}\left[\omega\in\Omega:\sup_{|x|<K}\left|y^\omega_{L+1}\left(\theta+\frac{x}{L}\right)-(L+1)\theta-g^\omega_\theta-x\right|>\frac{\eta}{3}\right]=0
\end{align*}
Set 
\begin{align*}
 &\Omega_L=\left\{\omega\in\Omega: \sup_{|x|<K}\left|y^\omega_{L+1}\left(\theta+\frac{x}{L}\right)-(L+1)\theta-g^\omega_\theta-x\right|<\frac{\eta}{3}\right\},
\end{align*}
and, using the enumeration \eqref{enumEq1} resulting from lemma \ref{lem5}, for $\omega\in\Omega_L$ we get
$$|(n^\omega_L+i)\pi-(L+1)\theta-g^\omega_\theta-x_i^{\omega,L}|<\frac{\eta}{3}.$$
Using the minimality of definition for $x_0^{\omega,L}$, we have $|n^\omega_L\pi-(L+1)\theta-g^\omega_\theta|<\pi+\eta$, so set $\phi_\theta^{\omega,L}=n^\omega_L\pi-(L+1)\theta-g^\omega_\theta$, which provides
$$|k\pi+\phi^{\omega,L}_\theta-x_k^{\omega,L}|<\frac{\eta}{3}\qquad \forall \omega\in\Omega_L.$$
For $\omega\in\Omega_L$, using \eqref{appEq1} and above
\begin{align*}
\hspace{-1cm}\left|2L\left(\cos\left(\theta+\frac{x^{\omega,L}_k}{L}\right)-\cos\theta\right)+2(k\pi+\phi^{\omega,L}_\theta)\sin\theta\right|\\
\leq 2|k\pi+\phi^{\omega,L}_\theta-x^{\omega,L}_k|+O\left(\frac{K^2}{L}\right)<\eta.
\end{align*}
So
\begin{align}\label{thmEq1}
&\Ex_\omega\left[\left|e^{i \xi^{\omega}_{E_0,L}(f)}-e^{i\sum_k f(-2(k\pi+\phi^{\omega,L}_\theta)\sin\theta)}\right|\right]\nonumber\\
&\qquad\leq 2\mathbb{P}[\Omega_L^c]+\Ex_{\omega\in\Omega_L}\left[\left|e^{i \sum_k f(L(E^{\omega,L}_k-E_0))}-e^{i\sum_k f(-2(k\pi+\phi^{\omega,L}_\theta)\sin\theta)}\right|\right].
%&\qquad\leq 2\mathbb{P}[\Omega_L^c]+\Ex_{\omega\in\Omega_L}\left[\left|e^{i \sum_k \left[f(2L(\cos\left(\theta+\frac{x^{\omega,L}_k}{L}\right)-\cos\theta))-f(-2(k\pi+\phi^{\omega,L}_\theta)\sin\theta)\right]}-1\right|\right]\nonumber\\
%&\qquad\leq 2\mathbb{P}[\Omega_L^c]+\Ex_{\omega\in\Omega_L}\left[\sum_k\left|f\left(2L\left(\cos\left(\theta+\frac{x^{\omega,L}_k}{L}\right)-\cos\theta\right)\right)-f(-2(k\pi+\phi^{\omega,L}_\theta)\sin\theta)\right|\right]\nonumber\\
%&\qquad\leq 2\mathbb{P}[\Omega_L^c]+\epsilon\sup_{\omega\in\Omega_L}\#\{i:|x^{\omega,L}_i|<K \}\leq \mathbb{P}[\Omega_L^c]+\frac{2\epsilon K}{\pi}.
\end{align}
For the second part in the RHS of  \eqref{thmEq1} observe
\begin{align}\label{thmEq3}
&\Ex_{\omega\in\Omega_L}\left[\left|e^{i \sum_k f(L(E^{\omega,L}_k-E_0))}-e^{i\sum_k f(-2(k\pi+\phi^{\omega,L}_\theta)\sin\theta)}\right|\right]\nonumber\\
&\qquad\leq \Ex_{\omega\in\Omega_L}\left[\sum_k\left|f\left(2L\left(\cos\left(\theta+\frac{x^{\omega,L}_k}{L}\right)-\cos\theta\right)\right)\right.\right.\nonumber\\
&\qquad\qquad\qquad\qquad -f(-2(k\pi+\phi^{\omega,L}_\theta)\sin\theta)\bigg|\bigg]\nonumber\\
&\qquad\leq 2\mathbb{P}[\Omega_L^c]+\epsilon\sup_{\omega\in\Omega_L}\#\{i:|x^{\omega,L}_i|<K \}\leq \mathbb{P}[\Omega_L^c]+\frac{2\epsilon K}{\pi}.
\end{align}

Last line follows using 
\begin{align*}
 K>|x^{\omega,L}_p|&>||x^{\omega,L}_p-p\pi-\phi^{\omega,L}_\theta|-|p\pi+\phi^{\omega,L}_\theta||\\
\Rightarrow\qquad &|p\pi+\phi^{\omega,L}_\theta|\leq K+\frac{\eta}{3}\qquad  for~\omega\in\Omega_L\\
\Rightarrow\qquad &\sup_{\omega\in\Omega_L}\#\{i:|x^{\omega,L}_i|<K \}\leq \frac{2K}{\pi}
\end{align*}
for $\omega\in\Omega_L$.
Now using the subsequence $\{L_{n_k}\}_k$, it is clear that $\{\phi^{\omega,L_{n_k}}_\theta\}_k$ converges, so set
\begin{equation}\label{thmEq2}
 \phi^{\omega,L_{n_k}}_\theta\xrightarrow{k\rightarrow\infty}\tilde{g}^\omega_\theta.
\end{equation}
%where $\{\frac{g^\omega_\theta+a-\tilde{g}^\omega_\theta}{\pi}\}=0$. 
Combining \eqref{thmEq1}, \eqref{thmEq3} and \eqref{thmEq2} gives
$$\lim_{n\rightarrow\infty}\Ex_\omega\left[e^{i \xi^\omega_{E_0,L_n}(f)}\right]=
\Ex_\omega\left[e^{i \sum_k f(-2(k\pi+\tilde{g}^\omega_\theta)\sin\theta)}\right]$$
completing the proof.\\
\end{proof}
\begin{remark}
For the case of $\tilde{H}^\omega_L$, above proof follows exactly by putting $g^\omega_\theta$ to be zero. This is the case because of the equation \eqref{remLab1Eq1} in remark \ref{remLabel1}.
\end{remark}

\subsection*{Acknowledgement}
The authors would like to thank Fumihiko Nakano for helpful comments and suggestions. The author Dhriti R. Dolai is supported by N\'{u}cleo Milenio de F\'{i}sica Mathem\'{a}tica, ICM grant  RC120002. The author Anish Mallick is partially supported by IMSc Project 12-R\&D-IMS-5.01-0106.

%%%%%%%%%%% The bibliography starts:

\noindent\makebox[\linewidth]{\rule{\linewidth}{0.4pt}}
{\small Anish Mallick\\ \texttt{anishm@imsc.res.in}\\Institute: The Institute of Mathematical Sciences, Chennai, India\\\\
Dhriti Ranjan Dolai\\ \texttt{dhdolai@mat.uc.cl}\\Institute: Pontificia Universidad Cat\'{o}lica de Chile, Santiago de Chile}

\end{document}